\documentclass[12pt]{amsart}
\usepackage{amsmath,amssymb}
\usepackage[shortlabels]{enumitem}
\def\d{\delta}
\def\D{\mathbb D}
\def\C{\mathbb C}
\def\R{\mathbb R}

\newtheorem{thm}{Theorem}
\newtheorem{prop}[thm]{Proposition}

\newtheorem{definition}[thm]{Definition}
\newtheorem{claim}[thm]{Claim}
\renewcommand{\Im}{\operatorname{\rm{Im}}}
\renewcommand{\Re}{\operatorname{\rm{Re}}}
\newcommand{\eps}{\varepsilon}
\newenvironment{proof*}{\vskip 2mm\noindent {}}{\hfill $\Box$ \vskip 2mm}

\title[Optimal bounds for the Kobayashi distance]
{Optimal bounds for the Kobayashi distance near $\mathcal C^2$-smooth boundary points}

\author{Nikolai Nikolov}
\address{N. Nikolov\\Institute of Mathematics and Informatics\\Bulgarian Academy
of Sciences\\Acad. G. Bonchev 8, 1113 Sofia, Bulgaria
\vspace{1mm}
\newline Faculty of Information Sciences\\
State University of Library Studies and Information Technologies\\
Shipchenski prohod 69A, 1574 Sofia,
Bulgaria}

\email{nik@math.bas.bg}

\author{Pascal J. Thomas}
\address{P.J. Thomas\\
Institut de Math\'ematiques de Toulouse; UMR5219 \\
Universit\'e de Toulouse; CNRS \\
UPS, F-31062 Toulouse Cedex 9, France}

\email{pascal.thomas@math.univ-toulouse.fr}

\thanks{  
\\
The  first named author was partially supported by the Bulgarian National
Science Fund, Ministry of Education and Science of Bulgaria under contract KP-06-N82/6.
}

\begin{document}


\begin{abstract} It is shown that the optimal upper and lower bounds for the Kobayashi distance
near $\mathcal C^{2,\alpha}$-smooth strongly pseudoconvex boundary points obtained in \cite{KNO}
remain true in the general $\mathcal C^2$-smooth strongly pseudoconvex setting. In fact, the upper bound
is extended to the general $\mathcal C^{1,1}$-smooth case. We also give upper and lower bounds
for the Kobayashi distance near non-semipositive boundary points.
\end{abstract}

\maketitle

\section{Introduction}

\subsection{Motivations.}

The study of function theory on domains in the complex plane, or more generally in Riemann surfaces, largely
rests on Riemann's mapping theorem (for simply connected domains in $\C$ distinct from it) and more generally
on the uniformization theorem, to pull objects back to the unit disc.  Even inside the unit disc, it is
often crucial to use the transitivity of the automorphism group to go back to a neighborhood of the origin.

As holomorphy in several complex variables is defined by restricting to one-complex-dimensional objects,
it is not surprising that many invariants are obtained, and many phenomena explained, by studying
maps from the unit disk into whichever complex space is being considered. This can be quantified
 by  the Kobayashi-Royden pseudo-metric: given a tangent vector $v$ at a point $p$, it will be of length no more than $1$
if and only if there exists a map $\varphi$ from the unit disk to the space with $\varphi(0)=p$
and $\varphi'(0)=v$.  Solving the natural extremal problem (see the precise definition in subsection \ref{invq})
 yields a Finsler metric, which is decreasing under holomorphic maps, thus invariant under biholomorphisms; it is
 the largest such invariant metric which coincides with the Poincar\'e metric on the unit disc.

By measuring the length of curves and passing to the infimum, one defines a pseudo-distance called
the Kobayashi distance, and a notion of hyperbolicity--that this pseudo-distance be an actual one--which
greatly generalizes one-variable hyperbolicity.

The Kobayashi and other invariant (pseudo-)distances are a powerful tool in complex analysis and geometry.
In spite of their elegant definition,
they are almost never explicitly computable, and quite often geodesic curves for the metric (if they exist at all)
cannot be determined either.
However, obtaining estimates about
their behavior allows us to use the methods of metric geometry to  the geometry of domains in $\C^n$.
Using Gromov hyperbolicity \cite{BB} or visibility \cite{BNT}, one can prove results about the extension
to the boundary of holomorphic maps defined on the interior of domains \cite{BB}, \cite{BNT}, \cite{BGNT}.

The first step to estimate distances is to restrict the domain of the (almost-)geodesics to
some neighborhood of the points considered, ideally to control its Euclidean length by the Euclidean
length of the shortest paths (in the usual sense) between the two points concerned.  This is known as the
Gehring-Hayman  property and is a topic of current study \cite{KNT}, \cite{LPW1}, \cite{LPW2},
with several natural questions still open, such as whether all bounded (strictly) pseudoconvex domains
enjoy it.

Many estimates are known about infinitesimal invariant {\it metrics} near boundary
points.  It is harder to obtain local estimates of the Kobayashi {\it distance}, but once gained,
they shed some light on where geodesics may not go, under penalty of excessive Kobayashi length.

\subsection{Overview.}
Our focus in this paper is obtaining local estimates of the Kobayashi  distance
near $\mathcal C^2$-smooth boundary points, in the  strongly pseudoconvex case (when all eigenvalues
of the Levi form are positive in a neighborhood of the point) and in the non-semipositive case
(when at least one eigenvalue of the Levi form is negative in a neighborhood of the point).

Recently \cite{KNO} very precise upper and lower estimates were obtained for the Kobayashi distance in terms
of explicit Euclidean geometric quantities for the $\mathcal C^{2,\alpha}$-smooth, strongly pseudoconvex
domains (see Proposition \ref{kno} below), using, as in \cite{KNT}, methods of dilation and the study of complex geodesics to reduce the question to the case of the ball.
We provide distance estimates (Theorem \ref{thm}) which are critical both near small and large distances
and involve quantities of the same type for the lower and upper bounds, in the case
of  merely $\mathcal C^2$-smooth strongly pseudoconvex boundary points.  In particular, they generalize
the estimates which are at the heart of  \cite{BB}.
Besides the reduction in the regularity required, our proof is shorter and only resorts 
to more elementary methods.

In addition, we provide estimates in the case of pseudoconcavity
(Theorems \ref{nonlow} and \ref{nonup}), which has recently
proved to be of interest in the context of the question of under which conditions visibility
implies pseudoconvexity \cite{NOT}. The presence of only one
negative eigenvalue of the Levi form is enough to change the situation completely, which may sound surprising,
but not when one remembers that it was so already for the infinitesimal metrics \cite{DNT}.
We note that the distance induced by the Sibony metric will satisfy similar estimates because
of the inequalities between the respective infinitesimal metrics, see \cite{DNT} and references therein.

\subsection{Definitions and notations.}
\label{notations}
We will write $x \lesssim y$, or equivalently $y\gtrsim x$, when $x\le C y$ for some constant $C$, and
$x\asymp y$ when $x \lesssim y$ and $x\gtrsim y$. We write
$\langle Z,W \rangle:= \sum_{j=1}^n Z_j \overline W_j$ for the usual Hermitian inner product in $\C^n$,
and $|Z|^2:=\langle Z,Z \rangle$.

The estimates that can be obtained depend on the regularity of the boundary of $D$.
Recall that we say that $\partial D$ is \emph{Dini-smooth} if its unit normal is a Dini-continuous
function of the point. Thus for any $\alpha>0$, if $\partial D$ is  $\mathcal C^{1,\alpha}$-smooth,
it is Dini-smooth, and if the latter property is verified, $\partial D$ is $\mathcal C^1$-smooth.

Let $z \in D$. Then $\d_D (z) =\mbox{dist}(z,\partial D)= \inf_{q\in \partial D} |z-q|$.
When there is a unique  point $q \in \partial D$ such that $|z-q|= \d_D(z)$, we write $q=\pi(z)$.
In that case, $\delta_D$ is differentiable at $z$.
This happens in particular when $\partial D$ is $\mathcal C^{1,1}$-smooth and $\d_D(z)$ is small enough.

At various places, we will need a notion of directional distance to the boundary:
\begin{equation}
\label{dirdist}
\d_D(z;X)=\sup\{\rho>0:z+\lambda X\in D\hbox{ for all } \lambda \in \C \mbox{ s.t. }|\lambda|<\rho\}.
\end{equation}

For $p\in\partial D$ be a point 
 at which $\partial D$ is differentiable, let $n_p$ stand for the inner unit
normal vector to $\partial D$.

We often need to denote the component  of a vector along the complex normal
 to $\partial D$ at the closest point $\pi(z)$ to $z.$ \begin{definition}
\label{subz}
Let $D$ be a domain with differentiable boundary,
  $X\in\C^n,$ and $z\in D$ such that $\pi(z)$ is well-defined. We set
\[
X_z=2\langle X,\bar\partial\delta_D(z)\rangle = \langle X,n_{\pi(z)}\rangle.
\]
\end{definition}
Our results will be stated in terms of the following auxiliary quantities.
\begin{eqnarray}
\label{ADBD}
B_D(z,w)&=&|(z-w)_z|+|z-w|^2+|z-w|\sqrt{\delta_D(z)},
\\
\notag
 A_D(z,w)&=&\frac{B_D(z,w)}{\sqrt{\d_D(z)\d_D(w)}}.
\end{eqnarray}
By \cite[Remark 7]{KNO}, $A_D(z,w)\asymp A_D(w,z)$ (see also \eqref{8} below).

\subsection{Invariant quantities.}
\label{invq}

Let $D$ be a domain in $\C^n,$ $z,w\in D$ and $X\in\C^n.$ Denote by $\D$ the unit disc.
The Kobayashi-Royden (pseudo)metric $\kappa_D$ and the Lempert function are defined as follows:
$$\kappa_D(z;X)=\inf\{|\alpha| :\exists\varphi\in\mathcal O(\D,D), \varphi(0)=z, \alpha\varphi'(0)=X\},$$
$$l_D(z,w)=\inf\{\tanh^{-1}|\alpha|:\exists\varphi\in\mathcal O(\D,D)
\hbox{ with }\varphi(0)=z,\varphi(\alpha)=w\}.$$
The Kobayashi (pseudo)distance $k_D$ is the largest pseudodistance not exceeding $l_D.$
It turns out that $k_D$ is the integrated form of $\kappa_D,$ that is,
$$k_D(z,w)=\inf_{\gamma}\int_0^1\kappa_D(\gamma(t);\gamma'(t))dt,$$
where the infimum is taken over all absolutely continuous curves
$\gamma:[0,1]\to D$ such that $\gamma(0)=z$ and $\gamma(1)=w.$

\subsection{Previous results.}

The main problem about the boundary behavior of $k_D(z,w)$ and $l_D(z,w)$ is the case when $z$
and $w$ approach the same point.

In \cite[Theorem 7]{NA}, the following upper bound for $k_D$ near Dini-smooth boundary
points of $D$ was proved:
$$k_D(z,w)\le\log\left(1+\frac{2|z-w|}{\sqrt{\d_D(z)\d_D(w)}}\right).$$

Much more is known in the strongly pseudoconvex case.

In a small enough neighborhood of a strongly pseudoconvex point, there is a sub-Riemannian structure on $\partial D$
whose horizontal curves are the complex-tangential ones, see e.g. \cite[p. 505]{BB}.
Denote by $cc_D$ the induced Carnot-Carath\'eodory distance.

From \cite[(4) and (6)]{KNO}, we see that
\begin{equation}\label{cc}
(1+A_D(z,w))\d_D(z)^{1/2}\d_D(w)^{1/2}\asymp cc_D(\pi(z),\pi(w))^2
+\max(\d_D(z),\d_D(w)),
\end{equation}

So, assuming $|X_z|=|X|$ if $\d_D(z)\gtrsim 1,$
the well-known Balogh-Bonk estimate \cite[Corollary 1.3]{BB} can  also be read as follows.

\begin{prop}\label{bb} For any strongly pseudoconvex domain $D$ there exists $C>0$ such that
\begin{equation}\label{1}\log(1+A_D)-C\le k_D\le\log(1+A_D)+C.
\end{equation}
\end{prop}

Using a different approach based on complex geodesics, \cite[Corollary 8]{KNO} refines
\eqref{1} (when $A_D\lesssim 1$) in the $\mathcal C^{2,\alpha}$-smooth case.

\begin{prop}[\cite{KNO}]

\label{kno}
If $p\in\partial D$ is strongly pseudoconvex
and $\mathcal C^{2,\alpha}$-smooth, then there exist $0<c<C$ such that near $p$ one has that
\begin{equation}
\label{2}
\log(1+cA_D)\le k_D\le
\end{equation}
\begin{equation}
\label{3}
l_D\le\log(1+CA_D).
\end{equation}
\end{prop}

\section{Results}

There is a partial converse to Proposition \ref{kno}.

\begin{prop}\label{spc} If \eqref{2} holds near a $\mathcal C^2$-smooth point $p\in\partial D,$
then $p$ is strongly pseudoconvex.
\end{prop}
This is proved in Section \ref{pftup}.

Our goal is to extend Propositions \ref{bb} and \ref{kno}.

\begin{thm}
\label{thm}
\eqref{2} and \eqref{3} hold near any
strongly pseudoconvex point.
\end{thm}
The proof of \eqref{2} is given in Section \ref{pftup}.

In fact, \eqref{3} is an easy consequence of the following general statement.

For  $p\in\partial D$ a $\mathcal C^{1,1}$-smooth point, define
$$
\chi_{D,p}=\limsup_{\partial D\ni p',p''\to p}\frac{|n_{p'}-n_{p''}|}{|p'-p''|}<\infty.
$$
Notice that when $p\in \partial D$ is a $\mathcal C^2$-smooth point, $\chi_{D,p}$ is equal
to the maximum of the absolute value of the principal curvature of $\partial D$ at $p$. 
Also, writing $h_r(z)$ for a homothety of ratio $r>0$, 
for $p\in\partial D$, the tangent and normal directions
to $\partial \left( h_r(D)\right)$ at $h_r(p)$ remain the same as those to $\partial D$ at $p$,
so $n_{h_r(p)}=n_p$ and $\chi_{h_r(D),h_r(p)}=r^{-1}\chi_{D,p}$.

Therefore the new quantity
$$A_{D,p}(z,w)=\frac{|(z-w)_z|+\chi_{D,p}|z-w|^2+|z-w|\sqrt{\chi_{D,p}\delta_D(z)}}
{\sqrt{\d_D(z)\d_D(w)}}
$$
is invariant under homothety: $A_{h_r(D),h_r(p)}(h_r(z),h_r(w))=A_{D,p}(z,w)$.

\begin{thm}
\label{smooth}
There exists $C_0>0$ such that for any domain $D$ (in any $\C^n$) and any
$\mathcal C^{1,1}$-smooth point $p\in\partial D$ with $\chi_{D,p}>0$ one may
find a neighborhood $U$ of $p$ such that
\begin{equation}\label{4}
k_D(z,w)\le\log(1+C_0 A_{D,p}(z,w)),\quad z,w\in D\cap U.
\end{equation}
\end{thm}

It follows immediately from the proof of Theorem \ref{smooth}, given
in Section \ref{pftsm} below,
that if $\chi_{D,p}=0,$ then the same holds replacing $\chi_{D,p}$
in $A_{D,p}$ by any $\varepsilon>0.$

Recall now that the Carath\'eodory (pseudo)distance is defined by
$$c_D(z,w)=\sup\{\tanh^{-1}|f(w)|: f\in\mathcal O(D,\D), f(z)=0\}.$$

One can deduce from Theorem \ref{thm}  its global version corresponding to
\cite[Theorem 1]{KNO}. We skip the proof because the arguments are the same as in \cite{KNO}.

\begin{prop} For any strongly pseudoconvex domain $D$ there exist $0<c<C$ such that
$$
\log(1+cA_D)\le c_D\le k_D\le l_D\le\log(1+CA_D).
$$
\end{prop}

Finally, we will consider the opposite case to strong pseudoconvexity.
Recall that $s: \C^n \longrightarrow \R$ is a \emph{defining function} for $D$
if $D=\{z\in\C^n: s(z)<0\}$. We call a $\mathcal C^2$-smooth point
$p\in\partial D$ {\it non-semipositive} if, for $s$ a defining function of $D$
(it does not depend on which one), $\partial{\overline\partial}s (X,\overline X)(p)<0$
for some vector $X\in T_p^{\C}(\partial D).$

The \emph{signed distance} to $\partial D$
is the special defining function given by $r_D=-\d_D$ on $D$ and $r_D=\d_D$ on $\C^n\setminus D$.

Set
\begin{equation}
\label{defHD}
H_D(z,w)=\frac{|(z-w)_z|}{\sqrt{|(z-w)_z|}+|z-w|+\sqrt{\d_D(z)}} + |z-w|.
\end{equation}
Note that 
\begin{multline}
\label{BH}
\frac{B_D(z,w)}{\sqrt{|(z-w)_z|}+|z-w|+\sqrt{\d_D(z)}} \le H_D(z,w)
\\
 \le \frac32 \frac{B_D(z,w)}{\sqrt{|(z-w)_z|}+|z-w|+\sqrt{\d_D(z)}},
\end{multline}
with $B_D$ as in \eqref{ADBD}.

Notice as well that $H_D(z,w)\asymp H_D(w,z)$.
Indeed, it already follows from \eqref{8} that $B_D(z,w)\asymp B_D(w,z)$.
Now consider the denominator
$\sqrt{|(z-w)_z|}+|z-w|+\sqrt{\d_D(z)}$.

From \eqref{zcomp}, $\sqrt{|(z-w)_z|}\le \sqrt{|(z-w)_w|}+ |z-w|$,
and from  \eqref{diffbound},
\[
\d_D(z)\le \d_D(w)+ |(z-w)_z| + |z-w|^2,
\]
so that
\[
\sqrt{|(z-w)_z|}+|z-w|+\sqrt{\d_D(z)} \le
3 \left( \sqrt{|(z-w)_w|}+|z-w| + \sqrt{\d_D(w)}\right),
\]
then switching the variables $z, w$, we get the result.

The quantity $H_D$ provides a lower bound which holds regardless of pseudoconvexity hypotheses (and is of course
much smaller than the one in \eqref{2}, hence far from optimal in the strongly pseudoconvex case).

\begin{thm}
\label{nonlow}
Let $p$ be a $\mathcal C^{1,1}$-smooth point in $\partial D,$
then for $z,w$ close enough to $p$, $k_D(z,w) \gtrsim H_D(z,w)$ .
\end{thm}

In particular,
$k_D(z,w)\gtrsim |z-w|+|\sqrt{\d_D(z)}-\sqrt{\d_D(w)}|$ (see Proposition \ref{lower}).

\begin{thm}
\label{nonup}
Let $p$ be a non-semipositive point in $\partial D,$
then for $z,w$ close enough to $p$, $k_D(z,w) \lesssim H_D(z,w)$ .
\end{thm}

In particular, near any non-semipositive point, $k_D(z,w)\asymp H_D(z,w)$ and
$k_D(z,w)\lesssim\sqrt{|(z-w)_z|}+|z-w|$.

Theorems \ref{nonlow} and \ref{nonup} are proved in Section \ref{pftnon}.

\section{Proof of Theorem \ref{thm}}
\label{pftup}

Henceforth, we will always assume that $z,w\in D$ are close enough to $p\in\partial D.$

We start with the proof of the easier converse direction.
\smallskip

\begin{proof*}{\it Proof of Proposition \ref{spc}.} From 
$\eqref{2}$ follows that $l_D(z,w)\ge k_D(z,w)\ge\log(1+cA_D(z,w))$, and 
an easy calculation shows that 
\[ 
\lim_{t\to0^+} \frac1t A_D(z,z+tX) = \frac{|X_z|}{\d_D(z)}+ \frac{|X|}{\sqrt{\d_D(z)}}.
\]
It follows by \cite[Proposition 2]{NP} that
$$
\kappa_D(z;X) = \lim_{t\to0^+} \frac1t l_D(z,z+tX)\ge c\frac{|X|}{\sqrt{\d_D(z)}}.
$$

We will show that if this lower estimate on $\kappa_D$ holds over a neighborhood of $p$, then $p$ is strongly pseudoconvex.
First, recall that we can find a local biholomorphism $\Phi$ on a neighborhood $U$ of $p$ so that 
at $p$ the defining function of $\Phi^{-1} (D\cap U)$ has a Hessian which reduces to the Levi form.

The argument is the same that shows that a strongly pseudoconvex domain is locally convexifiable. 
Using appropriate Euclidean coordinates and the Implicit Function Theorem, we may assume $p=0$ and that in some
neighborhood $U$ of $0$, $D\cap U=\{ z\in U: \rho(z)<0\}$, with 
\begin{multline*}
\rho(z)= \Re z_1 + F(z'), \\ 
\mbox{where } z=(z_1,z')\in \C\times\C^{n-1}, 
F\in \mathcal C^2( U \cap \C^{n-1}, \R), DF(0)=0.
\end{multline*}
Let 
\[
\Phi(z_1,z')= \left( z_1- \sum_{2\le j,k\le n} \frac{\partial F}{\partial z_j\partial z_k}(0) z_j z_k, z'\right) .
\]
Notice that the cross terms are repeated in the sum above. By Taylor's formula,
\[
\rho \circ \Phi (z) =  \Re z_1 + \Re\left(\sum_{2\le j,k\le n} \frac{\partial F}{\partial z_j\partial \bar z_k}(0) z_j \bar z_k \right)
+ o(|z'|^2).
\]
Suppose that $\partial D$ is not strongly pseudoconvex at $0$, then the same holds for $\Phi(U\cap \partial D)$,
and there is a vector $X \in \{0\}\times\C^{n-1}$ such that the Hessian $\mathcal H (\rho \circ \Phi ) (X) \le 0$.
Without loss of generality, assume $X=(0,1,0,\dots,0)$, take $z_\delta=(-\delta, 0,\dots,0)$, then
$\rho \circ \Phi (z_\delta+\lambda X) \le \frac12 \delta + o(|\lambda|^2)$ for $\delta>0$ small enough,
and
\begin{multline*}
\kappa_{D} (\Phi (z_\delta); D\Phi(X)) \le
\kappa_{D\cap U} (\Phi (z_\delta); D\Phi(X))
\\
= \kappa_{\Phi^{-1} (D\cap U)} (z_\delta; X) \le \frac{1}{\d_{\Phi^{-1} (D\cap U)}(z;X)} = o\left( \frac1{\sqrt \delta}\right),
\end{multline*}
while $\delta_D(\Phi (z_\delta)) \asymp \delta$, $D\Phi(X)\asymp 1$. 
This is a contradiction to the lower estimate on $\kappa_D$.
\end{proof*}

We  need to study the behavior of the quantity $A_D$ defined in
\eqref{ADBD} under biholomorphisms and switching the variables. To do so, it is convenient to
introduce the following slightly more general quantity:
for any $\mathcal C^{1,1}$-smooth defining function $s$ of $D$ near $p\in\partial D,$
define $A^s_D$ and $B^s_D$ as in \eqref{ADBD}, replacing $\d_D$ by $-s$ throughout.

\begin{prop}\label{c2} $A^s_D\asymp A_D.$
\end{prop}

\begin{proof} Note that $s=h\d_D,$ where $h<0$ is Lipschitz near $p.$ Then
$\bar\partial s=h\bar\partial \d_D +O(\d_D)$ which easily implies that $A^s_D\asymp A_D.$
\end{proof}

\begin{prop}\label{c3} Let $D,G\Subset\C^n$ and $f:\overline D\to\overline G$ be a biholomorphism.
If $\partial D$ is $\mathcal C^{1,1}$-smooth, then for $z,w \in D$,
$$A_D(z,w)\asymp A_G(f(z),f(w)).$$
\end{prop}

\begin{proof} Note that $s:=-\d_D\circ f^{-1}$ is a defining function of $G$ and
\begin{multline}
\label{fscal}
\langle f(w)-f(z),\bar\partial s(f(z))\rangle
\\
=
\langle Jf(z)(z-w)+O(|z-w|^2),((Jf(z))^{-1})^*\bar\partial \d_D(z)\rangle
\\
=\langle z-w,\bar\partial \d_D(z)\rangle+O(|z-w|^2)
\end{multline}
(the vectors are column vectors).

Then \eqref{fscal} and $|f(z)-f(w)|\asymp |z-w|$ imply that
$A_D(z,w)\asymp $
\newline $A^s_G(f(z),f(w)).$
By Proposition \ref{c2}, we are done.
\end{proof}
\smallskip

\noindent{\it Proof of Theorem \ref{thm}.}  As above, there exist a neighborhood $U$ of $p$ and a
biholomorphism $\Phi:U\to\Phi(U)$ such that $G:=\Phi(D\cap U)$ is a strongly convex domain.

By \cite[Lemma 20]{KNO} (see also Proposition \ref{c3}), there is a neighborhood $V$ of $p$
such that for $z,w \in D\cap V$,
$$
A_D(z,w)=A_{D\cap U}(z,w)\asymp A_G(\Phi(z),\Phi(w)).
$$

On the other hand, by \cite[Theorem 1.1]{NT1}, again for $z,w \in D\cap V$, reducing $V$ as needed,
$$
k_D\le k_{D\cap U}<k_D+C,\quad k_{D\cap U}\asymp k_D.
$$

So we may replace $D$ by $G,$ which we again relabel $D$.

Then $c_D=k_D=l_D$ by Lempert's theorem and so \eqref{3} follows from Theorem \ref{smooth}.

Now we will prove \eqref{2}. The case $A_D(z,w)\gtrsim 1$ is
covered by the LHS inequality in \eqref{1}.

Further, since $A_D(z,w)\asymp A_D(w,z),$ we may also assume that
$\d_D(w)\ge\d_D(z).$
Then, by \cite[p.~633]{NM},
$$
k_D(z,w)\ge\frac{1}{2}\log\left(1+\frac{1}{\d_D(z;z-w)}\right).
$$
Since $D$ can be placed between two balls having $\pi(z)\in\partial D$ as a
boundary point and radii independent of $\pi(z),$ it follows that
$$
\frac{1}{\d_D(z;z-w)}\asymp\hat A_D(z,w):=
\frac{|(z-w)_z|}{\d_D(z)}+\frac{|z-w|}{\sqrt{\d_D(z)}}.
$$
Note that \cite[Proposition 22]{KNO} implies that
$$
\hat A_D(z,w)\asymp\frac{B_D(z,w)}{\d_D(z)}\ge A_D(z,w).
$$

So $k_D(z,w)\gtrsim A_D(z,w)$ if $A_D(z,w)\lesssim 1.$ \qed

\section{Proof of Theorem \ref{smooth}}
\label{pftsm}

To begin with, we need a technical result from \cite{NT2} about the behaviour of the distance function.
Write $x\cdot y := \sum_{j=1}^m x_j y_j$ for the Euclidean scalar product in $\R^m$.
When $\C^n$ is identified with $\R^{2n}$ in the usual way, this means that
$Z\cdot W = \Re \langle Z,W \rangle$.

\begin{prop}{\cite[Proposition 8]{NT2}}
\label{c1}
 Let $p$ be a $\mathcal C^{1,1}$-smooth boundary point
of a domain $D$ in $\mathbb R^m.$ Then the signed distance $r_D$ to $\partial D$ is a
$\mathcal C^{1,1}$-smooth function near $p.$ Moreover, for $x, y$ near $p$ one has that
\begin{multline}
|r_D(x)-r_D(y)-\nabla r_D(x)\cdot (x-y)|
\\
\le(\chi_{D,p}/2+o(1))(|x-y|^2-(r_D(x)-r_D(y))^2).
\end{multline}
\end{prop}

Whenever $\chi_{D,p}=\chi_0<1$, this implies
\begin{equation}
\label{diffbound}
|\d_D(z)-\d_D(w)| \le |\Re X_z|+(\chi_0/2+o(1))|X|^2\le B_D(z,w).
\end{equation}

We turn to the proof of Theorem \ref{smooth}, part (a).
Both the right and left hand sides of the inequality in \eqref{4} are invariant under homothety; 
so applying a homothety, which only changes the neighborhood on which the estimates
will hold, we may assume that $\chi_{D,p}$ is equal to some number $\chi_0\in(0,1).$

Also note that replacing $A_{D,p}$ with $A_D$, we only modify \eqref{4} by a constant
depending on $\chi_0$.

Let $E_{q,r}$ denote the ball with center
$q_\ast:=q+rn_q$ and radius $r.$ With our choice of $\chi_0$,
there exists a neighborhood
$V$ of $p$ such that for $q\in\partial D$ near $p$, $E_{q,1}\cap V\subset D$.

Set $X=z-w$, so that $B_D(z,w)=|X_z|+|X|^2+|X|\sqrt{\delta_D(z)}$.  

\begin{claim} 
\label{Bsym}
For $z,w$ close enough to $p$,
 $B_D(w,z) \asymp B_D(z,w)$, and therefore $A_D(w,z) \asymp A_D(z,w)$.
\end{claim}

\begin{proof}
By \eqref{diffbound}, 
\[
\delta_D(w)\le \delta_D(z)+|X|^2 + |X_z|, \mbox{ so }
\]
\begin{multline*}
|X|\sqrt{\delta_D(w)} \le |X|\sqrt{\delta_D(z)}+|X|^2 + |X|\sqrt{|X_z|} 
\\
\le |X|\sqrt{\delta_D(z)}+|X|^2 + \sqrt{B_D(z,w)} \sqrt{B_D(z,w)},
\end{multline*}
and since, using  $|\pi(z)-\pi(w)| \le (\chi_0+o(1)) |z-w|$,
\begin{equation}
\label{zcomp}
|X_w|-|X_z|\le |X_w-X_z| = \left| \langle X, n_{\pi(z)}- n_{\pi(w)}
\rangle \right|\le(\chi_0+o(1))|X|^2\le|X|^2,
\end{equation}
 we finally obtain 
\[
B_D(w,z) \le |X_z| + 2 |X|^2 + |X|\sqrt{\delta_D(z)} + B_D(z,w) \le 3  B_D(z,w).
\]
Thus
\begin{equation}\label{8}
A_D(w,z)\le 3 A_D(z,w).
\end{equation}
\end{proof}

So we may assume for the remainder of the proof that $\d_D(w)\le\d_D(z)$ (note that this
is the opposite convention to the proof of Theorem \ref{thm}). It is enough to show the following.

\begin{claim}
\label{mainclaim}
For $z,w \in U\cap D$, $U$ a small enough neighborhood of $p$,
\begin{enumerate}[(a)]
\item
$k_D(z,w)\le 8A_D(z,w)$ if $2B_D(z,w)\le \sqrt{\d_D(z)\d_D(w)};$
\item
$k_D(z,w)<\log A_D(z,w)+16$ if $2B_D(z,w)>\sqrt{\d_D(z)\d_D(w)}.$
\end{enumerate}
\end{claim}

To prove (a), let $\gamma$ be the line segment from $z$ to $w$,
explicitly $\gamma(t)=z-tX,$ $t\in[0,1]$, and choose $U$ convex and small enough so that
for any $u, v\in\gamma \subset U$, $X\in\C^n$, $\d_{V\cap E_{\pi(v),1},X}(u) = \d_{E_{\pi(v),1},X}(u)$.
Recall that now $\pi(z)_\ast= \pi(z)+n_{\pi(z)}$. Then
\begin{multline}
\label{estgamma}
1-|\gamma(t)-\pi(z)_\ast|^2
= 1 - \left| \pi(z)+ \delta_D(z) n_{\pi(z)} - tX - (\pi(z)+  n_{\pi(z)})\right|^2
\\
=1 - \left| (1- \delta_D(z)) n_{\pi(z)} + t X \right|^2
\\
\ge 2\d_D(z)-\d^2_D(z)-|X|^2-2(1-\d_D(z))|\mbox{Re} X_z|
\\
\ge 3\d_D(z)/2-2B_D(z,w)\ge\d_D(z)/2,
\end{multline}
hence by the triangle inequality $\gamma(t)\in D$.
Furthermore,  
\[
1-|\gamma(t)-\pi(z)_\ast| \ge \frac{\d_D(z)}{2(1+|\gamma(t)-\pi(z)_\ast|)}  > \frac{\d_D(z)}{4},
\]
so since $\pi(z)_\ast$ is the center of the ball $E_{\pi(z),1}$,
$\d_D(\gamma(t))\ge \d_{V\cap E_{\pi(z),1}}(\gamma(t)) =
\d_{E_{\pi(z),1}}(\gamma(t))>\d_D(z)/4.$

Since $\gamma$ is a competitor for the Kobayashi distance,
$k_D(z,w)\le \int_0^1 \kappa_D(\gamma(t);X)dt $ and there exists $u =\gamma(t_0)$ such that
$$
k_D(z,w)\le\kappa_D(u;X).
$$
Since
 $V\cap E_{\pi(u),1} \subset D$, so that  $ \kappa_D(u;X)\le \kappa_{V\cap E_{\pi(u),1}}(u;X)\le \frac{1}{\d_{V\cap E_{\pi(u),1}}(u;X)}
 =\frac{1}{\d_{ E_{\pi(u),1}}(u;X)}$.

To estimate this last quantity, we need to know for which $\lambda \in \C$ we have $u+\lambda X \in E_{\pi(u),1}$.
Computing as in \eqref{estgamma}, 
\[
1-|\pi(u)_* - (u+\lambda X )|^2 = 2 \d_D(u)- \d_D(u)^2 - 2(1-\d_D(u)) \Re(\lambda X_u)- |\lambda|^2 |X|^2.
\]
Assuming $\d_D(u)<1$, as we may, an explicit computation shows that for $|\lambda|\le \frac12 \left(\frac{|X_u|}{\d_D(u)}+\frac{|X|}{\sqrt{\d_D(u)}}\right)^{-1}$,  the above quantity is positive.  Thus, using $\d_D(u)\ge \d_D(z)/4,$
\begin{equation}
\label{ballest}
\kappa_D(u;X)\le
2\left(\frac{|X_u|}{\d_D(u)}+\frac{|X|}{\sqrt{\d_D(u)}} \right) \le
\frac{8|X_u|}{\d_D(z)}+\frac{4|X|}{\sqrt{\d_D(z)}}.
\end{equation}
Using that
$$
|X_u|-|X_z|\le|X| \cdot |u-z|\le|X|^2,
$$
we conclude that
$$
k_D(z,w)\le\frac{8B_D(z,w)}{\d_D(z)}\le 8A_D(z,w).
$$

We now prove (b).
The assumption on $A_D(z,w)$ implies
$4B_D(z,w)^2>\d_D(z)\d_D(w)$, which because of \eqref{diffbound} is
$\ge \d_D(z) (\d_D(z)-B_D(z,w))$, from which one deduces $3B_D(z,w)>\d_D(z).$

Let
\begin{equation}
\label{prime}
z'=\pi(z)+3B_D(z,w)n_{\pi(z)}\mbox{ and }
w'=\pi(w)+3B_D(z,w)n_{\pi(w)}\mbox{, then}
\end{equation}
\begin{equation}
\label{decompk}
k_D(z,w)\le k_D(z,z')+k_D(w,w')+k_D(z',w').
\end{equation}

Recall that $E_{q,1}\subset D$ for any $q\in\partial D$ near $p.$ Then
an explicit calculation of the invariant distances in the unit ball leads to
$$
k_D(z,z')\le k_{E_{\pi(z)},1}(z,z') \le \log\frac{\sqrt{\d_D(z')}}{\sqrt{\d_D(z)}}+o(1)
= \log\frac{\sqrt{3B_D(z,w)}}{\sqrt{\d_D(z)}}+o(1),
$$
$$k_D(w,w')\le k_{E_{\pi(w)},1}(w,w')\le\log\frac{\sqrt{\d_D(w')}}{\sqrt{\d_D(w)}}+o(1)
= \log\frac{\sqrt{3B_D(z,w)}}{\sqrt{\d_D(w)}}+o(1),$$
and hence
\begin{equation}
\label{vertbits}
k_D(z,z')+k_D(w,w')\le\log(3A_D(z,w))+o(1).
\end{equation}

To estimate $k_D(z',w'),$ set $X'=z'-w'.$
Writing $u'=z'-tX'$, we have as above that
\begin{multline*}
1-|u'-\pi(z)_\ast|^2 = 1 - \left| (1-3B_D)n_{\pi(z)} + t X' \right|^2
\\
 \ge(6+o(1))B_D(z,w)-|X'|^2-2|X'_z|.
\end{multline*}
By using \eqref{diffbound} and $|n_{\pi(z)}-n_{\pi(w)}| = O(|z-w|) = o(1)$, we find
\begin{multline}
|X'-X|= |(z-z')-(w-w')|
\\
=
\left| (3B_D(z,w)-\d_D(z))(n_{\pi(z)}-n_{\pi(w)}) +
(\d_D(z)-\d_D(w))n_{\pi(w)}\right|
\\
\le(3B_D(z,w)-\d_D(z))|n_{\pi(z)}-n_{\pi(w)}|+\d_D(z)-\d_D(w)
\\
\le(1+o(1))B_D(z,w).
\end{multline}
Since $|X|^2\le B_D(z,w),$ it follows from the triangle inequality that
\begin{equation}\label{5}
|X'|^2\le (1+O(B_D(z,w)^{1/2}))B_D(z,w) = (1+o(1))B_D(z,w).
\end{equation}

Also,
\begin{equation}\label{6}
|X_z'|\le|X_z|+|X-X'|\le(2+o(1))B_D(z,w)
\end{equation}
and so
$$1-|u'-\pi(z)_\ast|^2\ge (6+o(1))B_D(z,w)-|X'|^2-2|X'_z|\ge(1+o(1))B_D(z,w).$$
Hence $u'\in D$ and
\begin{equation}\label{7}
\d_D(u')\ge \d_{E_{\pi(u),1}}(u') \ge (1/2+o(1))B_D(z,w).
\end{equation}

Then as in the proof of (a), one may find $u'\in[z';w']$ such that
$$
k_D(z',w') \le \kappa_D (u'; X')
\le 2\left( \frac{|X'_{u'}|}{\d_D(u')}+\frac{|X'|}{\sqrt{\d_D(u')}} \right).
$$
Since $|X'_{u'}|-|X'_{z'}|\le|X'|^2$ and $X'_{z'}=X'_z,$  it follows by \eqref{5},
\eqref{6}, and \eqref{7} that
$$k_D(z',w')<2(6+\sqrt2+o(1))<16-\log3.$$

So (b) follows from \eqref{decompk} and \eqref{vertbits}.\qed

\section{Proofs of Theorems \ref{nonlow} and \ref{nonup}}
\label{pftnon}

\subsection{Proof of Theorem \ref{nonup}.}

Using a homothety (thus changing the neighborhoods if needed), we may assume that $\chi_{D,p}<1$.

Since $H_D(z,w)\asymp H_D(w,z)$, we may assume that $\d_D(z)\ge\d_D(w)$.

We will show that $k_D(z,w) \lesssim H_D(w,z)$.
By \cite[Proposition 1]{DNT}, near $p,$
\begin{equation}
\label{DNTest}
\hat\kappa_D(u;X)\lesssim\frac{|X_u|}{\sqrt{\d_D(u)}}+|X|.
\end{equation}

Define $z'$ and $w'$ as in \eqref{prime}.
\begin{claim}
\label{claimconcav}
\begin{enumerate}[(a)]
\item
$k_D(z,w)\lesssim\frac{B_D(z,w)}{\sqrt{\d_D(z)}}$ if $2B_D(z,w)\le\d_D(z);$
\item
$k_D(z,w)\le k_D(z,z')+k_D(w,w')+k_D(z',w')\lesssim\sqrt{B_D(z,w)}$ if $2B_D(z,w)>\d_D(z)$.
\end{enumerate}
\end{claim}
\smallskip

\noindent{\it Proof of Claim \ref{claimconcav}.} We follow the general scheme of the proof of Claim
\ref{mainclaim}, and reuse those of the arguments which only depend on the smoothness of $\partial D$.

To prove (a), we define $X$, $\gamma$, $u$ and $E_{\pi(u),1}$
as before and follow the proof until before \eqref{ballest}.  Instead of \eqref{ballest}, we use \eqref{DNTest},
and since $\d_D(u)\ge \frac14 \d_D(z)$ and $|X_u|-|X_z|\le |X|^2$,
\[
\hat\kappa_D(u;X)\lesssim  \frac{|X_z|+|X|^2}{\sqrt{\d_D(z)}}+|X|
\le  \frac{B_D(z,w)}{\sqrt{\d_D(z)}}.
\]

To prove (b), we do not need to use \eqref{diffbound} to bound
$B_D(z,w)$ from below, since $2B_D(z,w)>\d_D(z)$ by assumption.

To estimate $k_D(z,z')$ and $k_D(w,w')$, simply integrate
along the normal line to get as upper bound
$c\int_{\d_D(z)}^{3B_D(z,w)} t^{-1/2} dt \lesssim B_D(z,w)^{1/2}$
(the same computation for $w$).

Then we use $k_D(z',w') \lesssim |z'-w'|$, and
by the definition of $B_D(z,w)$,
\begin{multline*}
|z'-w'| \le |z'-z|+|z-w|+ |w-w'|
\\
\le 3 B_D(z,w) + \sqrt{B_D(z,w)} + 3 B_D(z,w) \lesssim \sqrt{B_D(z,w)}.
\end{multline*}
\qed

\noindent{\it End of the proof of $k_D(z,w) \lesssim H_D(w,z) $.}
In case (a), the assumption implies
$|(z-w)_z|+|z-w|^2 \le \d_D(z)/2$, so that
$\sqrt{|(z-w)_z|}+|z-w| \le \sqrt{\d_D(z)}$, and using \eqref{BH}
\[
\frac{B_D(z,w)}{\sqrt{\d_D(z)}} \le 2 H_D(z,w).
\]

In case (b), the assumption implies
\newline
$\sqrt{B_D(z,w)} \gtrsim \sqrt{|(z-w)_z|}+|z-w| + \sqrt{\d_D(z)}$,
so that
\[
\sqrt{B_D(z,w)} \lesssim \sqrt{B_D(z,w)}
\frac{\sqrt{B_D(z,w)}}{\sqrt{|(z-w)_z|}+|z-w| + \sqrt{\d_D(z)}}
\le H_D(z,w),
\]
using \eqref{BH} once more.
\qed

\subsection{Proof of Theorem \ref{nonlow}.}

We start with a first lower bound for $k_D$.

\begin{prop}\label{lower}
If $D$ is bounded, then $k_D\gtrsim F_D$
near a $\mathcal C^{1,1}$-smooth point $p\in\partial D,$ where
$F_D(z,w)=|z-w|+|\sqrt{\d_D(z)}-\sqrt{\d_D(w)}|.$
\end{prop}

\begin{proof}

We will use that $k_G$ is also the integrated form of the Kobayashi-Busemann (pseudo)metric
$\hat\kappa_G$ -- the largest (pseudo)metric not exceeding $\kappa_G$ with convex indicatrices.

Let $D_R:= B(0,R) \setminus \overline B (0,1)$. By \cite[Proposition 1]{DNT},
\begin{equation}
\label{DNTlowest}
\hat\kappa_{D_R}(u;X)\gtrsim\frac{|X_u|}{\sqrt{\d_{D_R}(u)}}+|X|.
\end{equation}

Given a $\mathcal C^2$ point $p$, by continuity of the second derivatives,
there exists a neighborhood $V$ of $p$ and a radius $R_0>0$ such that for any $p'\in V\cap \partial D$,
the ``outside ball'' $B_{p'}:= B(p'-R_0n_{p'}, R_0)$ verifies $B_{p'}\cap D=\emptyset$,
$B_{p'} \cap \overline D=\{p'\}$.
Using that $D$ is bounded, there exist $R>1$ and a neighborhood $U$ of $p$
such that for any $q\in U\cap\partial D,$ after a homothety and a translation, one has that
\begin{equation}
\label{incannu}
D \subset D_R\mbox{ and }q \in \partial B(0,1).
\end{equation}

Consider $z,w\in U$, close enough to $p$ so that  $q:=\pi(w) \in U$. We may assume that $\d_D(z)\ge\d_D(w)$.
We construct $D_R$ as above. Since $D_R$ and $D$ must have the same tangent space at $q$,
in our new coordinates, $q=\frac{w}{|w|}$.

Let $\gamma$ be a curve from $z$ to $w$ in $D$ such that
$\gamma(0)=z$ and $\gamma(1)=w$. Then
\begin{multline*}
\int_0^1 \hat \kappa_{D_R} ( \gamma(t); \gamma'(t) ) dt
\\
\gtrsim \int_0^1 \left(|\gamma'(t)| + \frac{\Re \langle \gamma'(t), \gamma(t) \rangle }
{(|\gamma(t)|^2-1)^{1/2}} \right)dt
=\int_0^1 |\gamma'(t)| dt + \frac12 \int_{|w|^2}^{|z|^2} \frac{dx}{\sqrt{x-1}}
\\ \ge|z-w|+\sqrt{|z|^2-1}-\sqrt{|w|^2-1}
\\
\ge|z-w|+\sqrt{\d_D(z)(2+\d_D(z))}-\sqrt{\d_D(w)(2+\d_D(w))}
\end{multline*}
$$\ge|z-w|+\sqrt{2\d_D(z)}-\sqrt{2\d_D(w)}.$$
So $k_D(z,w)\ge k_{D_R}(z,w)\gtrsim|z-w|+\sqrt{\d_D(z)}-\sqrt{\d_D(w)}.$
\end{proof}

It turns out that Proposition \ref{lower} is the ``real part'' version of  Theorem \ref{nonlow}.
Indeed, define $H_D^r$ by replacing $(z-w)_z$ by $\Re (z-w)_z$ in equation \eqref{defHD}:
\[
H_D^r(z,w)=\frac{|\Re (z-w)_z|}{\sqrt{|\Re (z-w)_z|}+|z-w|+\sqrt{\d_D(z)}} + |z-w|.
\]

\begin{prop}
\label{FDleqHD}
$F_D\asymp H_D^r\le H_D$ near any $\mathcal C^{1,1}$-smooth point $p\in\partial D.$
\end{prop}

\begin{proof}
The inequality $H_D^r(z,w)\le H_D(z,w)$ is
equivalent to $|\Re (z-w)_z|\le |(z-w)_z|.$

Further, observe that
$$
F_D(z,w)<\sqrt{\delta_D(z)}+\sqrt{\delta_D(w)}+|z-w|\mbox{ and }
$$
$$
\sqrt{|\Re (z-w)_z|}+|z-w|\le H_D^r(z,w).
$$
So, if $\sqrt{\delta_D(z)}+\sqrt{\delta_D(w)}\le |\Re (z-w)_z|+|z-w|,$ then $F_D(z,w)<2H_D^r(z,w).$

Otherwise,
\begin{multline*}
|\sqrt{\delta_D(z)}-\sqrt{\delta_D(w)}|=\frac{|\d_D(z)-\d_D(w)|}{\sqrt{\delta_D(z)}+\sqrt{\delta_D(w)}}
\le\frac{|\Re (z-w)_z|+C^2|z-w|^2}{\sqrt{\delta_D(z)}+\sqrt{\delta_D(w)}}\\
\le\frac{2|\Re (z-w)_z|}{\sqrt{|\Re (z-w)_z|}+|z-w|+\sqrt{\delta_D(z)}} + C^2|z-w|.
\end{multline*}
where the first inequality follows from Proposition \ref{c1}.

Thus $F_D(z,w)\le\max\{2,C^2+1\}H_D^r(z,w)$ in both cases.

Similarly, by Proposition \ref{c1},
\newline
$\sqrt{\delta_D(w)}\le \sqrt{\delta_D(z)}+ \sqrt{|\Re (z-w)_z|}+C|z-w|$, so
\begin{multline*}
|\sqrt{\delta_D(z)}-\sqrt{\delta_D(w)}|
\ge
\frac{|\Re (z-w)_z|-C^2|z-w|^2}{2\sqrt{\delta_D(z)}+
2\sqrt{|\Re (z-w)_z|}+C|z-w|}
\\
\ge C^{-1}H_D^r(z,w)-C|z-w|,
\end{multline*}
where $C\ge 2.$ This and $F_D(z,w)\ge |z-w|$ imply that $C(C+1)F_D(z,w)\ge H_D^r(z,w).$
\end{proof}

To get the full statement of Theorem \ref{nonlow}, we reduce ourselves to the case of the set difference of
two balls, using \eqref{incannu} and the considerations that precede it. This time we will choose coordinates
such that (after translation) $|z| \ge |w|$ and $\pi(z)= \frac{z}{|z|}$, and after an additional
rotation, $\frac{z}{|z|}=(1,0,\dots,0)$.  Notice that then $H_D(z,w)=H_{D_R}(z,w)$, since $n_z$
is the same whether it is taken with respect to $D$ or to $D_R$. Since $k_D(z,w) \ge k_{D_R}(z,w)$,
the estimate we seek reduces to the following.

\begin{prop}
\label{annulow}
Let $z=(1+\eps_1,0)\in D_R \subset \C\times \C^{n-1}$, let $w\in D_R$, with $|w|\le |z|$. Then
$k_{D_R}(z,w)\gtrsim H_{D_R}(z,w)$.
\end{prop}

\begin{proof}
Throughout this proof, we will simply write $D_R=D$.

A further rotation within the
plane $\{0\} \times \C^{n-1}$ allows us to write
\newline
$w=(1+\eps_2)(e^{i\eta} \cos \beta, \sin \beta, 0, \dots, 0)\in \C\times\C\times \C^{n-2}$,
with $0<\eps_2\le \eps_1$, and $0\le \beta \le \pi/2$. We only need the estimate for nearby
points, so that we may assume $\beta + (\eps_1 - \eps_2) + |\eta| \asymp |z-w| \ll 1$.

Then $\pi(z)= (1,0)$, $\pi(w)= (e^{i\eta} \cos \beta, \sin \beta,0, \dots, 0)$,
\begin{multline*}
(z-w)_z= (1+\eps_1) - (1+\eps_2)e^{i\eta}\cos \beta =
\\
 (\eps_1 - \eps_2) +
2  (1+\eps_2) (1- 2 \sin^2 \frac\eta2 ) \sin^2 \frac\beta2 + 2 (1+\eps_2) \sin^2 \frac\eta2
- i (1+\eps_2) \cos\beta \sin \eta.
\end{multline*}

Thus $|\Re (z-w)_z|\asymp  (\eps_1 - \eps_2) + \eta^2 + \beta^2$, $|\Im (z-w)_z|\asymp |\eta|$,
$|(z-w)_z|\asymp (\eps_1 - \eps_2) + |\eta| + \beta^2$.

This implies
\[
H_D(z,w) \asymp
\frac{(\eps_1 - \eps_2) + |\eta|+ \beta^2 + \beta \sqrt{\eps_1}}{ \sqrt{|\eta|}+ \beta +\sqrt{\eps_1}} ,
\]
while
\[
H_D^r(z,w)\asymp  |\eta|+ \beta + \sqrt{\eps_1}
-\sqrt{\eps_2}.
\]

When $|\Re (z-w)_z| \asymp |(z-w)_z|$, then $H_D^r(z,w) \asymp H_D(z,w)$, so we only need
to deal with the cases where this does not happen, i.e.
$(\eps_1 - \eps_2) +  \beta^2 \le c_0 |\eta|$, where $c_0\le 10^{-2}$ is an absolute constant,
the value of which will be fixed later.

With this new requirement,
\[
H_D(z,w) \asymp
\frac{ |\eta|+  \beta \sqrt{\eps_1}}{ \sqrt{|\eta|} +\sqrt{\eps_1}} .
\]
When $ |\eta|\le \eps_1$, this reduces to
$H_D(z,w) \asymp \beta+ \frac{ |\eta|}{\sqrt{\eps_1}}$;
when $ |\eta|\ge \eps_1$, this reduces to
$H_D(z,w) \asymp \beta+ \sqrt{|\eta|}$; so altogether, there are two absolute constants
$C_1<C_2$ so that for any choice of $c_0\le 10^{-2}$ we have
\begin{equation}
\label{est1hd}
C_1\left(  \beta+ \min \left( \sqrt{|\eta|}, \frac{ |\eta|}{\sqrt{\eps_1}} \right)\right) \le
H_D(z,w) \le C_2\left(  \beta+ \min \left( \sqrt{|\eta|}, \frac{ |\eta|}{\sqrt{\eps_1}} \right)\right).
\end{equation}

Let $\gamma$ be an absolutely continuous curve such that $\gamma(0)=z$, $\gamma(1)=w$.
Suppose that $\gamma$ is close enough
to minimize the integral of the Kobayashi-Busemann metric.
Since for any $t$, $k_D(z,\gamma(t)) \lesssim k_D(z,w)$,
and since $H^r_D \lesssim k_D \lesssim H_D$,
\[
|z-\gamma(t)| + \left| \sqrt{\d_D(z)} - \sqrt{\d_D(\gamma(t))}\right| \lesssim \beta+
\min \left( \sqrt{|\eta|}, \frac{ |\eta|}{\sqrt{\eps_1}} \right),
\]
so that for $\eps,\eta$ small enough, $\d_D(\gamma(t))= |\gamma(t)| -1$ and
$\d_D(\gamma(t)) \lesssim \eps_1 +\beta^2+ |\eta|$.

With the above restrictions on the values of $\d_D(\gamma(t))$, we can write
\[
\gamma(t)= (1+\rho(t)) \left( e^{i\theta_1(t)}\cos \alpha(t), \zeta(t)\sin \alpha(t) \right),
\]
where $\rho, \alpha, \theta_1$ are absolutely continuous real-valued functions,
and $\zeta(t)\in \C^{n-1}$, $|\zeta(t)|=1$ for all $t$. All
functions are also a.e. differentiable, and the condition on $\zeta$
implies $\Re \langle \zeta(t), \zeta'(t)\rangle =0$ for all $t$
where it makes sense.

We must have $\rho(t)>0$ for all $t$, $\rho(0)=\eps_1$,
$\rho(1)=\eps_2$, $\theta_1(0)=0$, $\theta_1(1)=\eta$,
$\zeta(1)=(1,0,\dots,0)$, $\alpha(0)=0$, $\alpha(1)=\beta$.

Write $\gamma^*(t):= \pi(\gamma(t))= \gamma(t)/|\gamma(t)|=\left( e^{i\theta_1(t)}\cos \alpha(t), \zeta(t)\sin \alpha(t) \right)$,
so that again $\Re \langle {\gamma^*}'(t),\gamma^*(t)\rangle =0$.
We compute
\begin{multline}
\label{deriv}
\gamma'(t) = \rho'(t) \gamma^*(t) + (1+\rho(t)) {\gamma^*}'(t)
\\
= \rho'(t) \gamma^*(t) +
(1+\rho(t)) \left(  e^{i\theta_1(t)}( i\theta_1'(t)\cos \alpha(t)-
\alpha'(t) \sin \alpha(t)) ,\right.
\\
\left. \zeta(t) \alpha'(t) \cos \alpha(t)  +  \zeta'(t) \sin \alpha(t)\right)
\end{multline}
so that
\[
|{\gamma^*}'(t)|^2 = |\theta_1'(t)|^2 \cos^2 \alpha(t) + |\alpha'(t)|^2 + |\zeta'(t)|^2 \sin^2 \alpha(t),
\]
\[
|\gamma'(t)|^2 = \rho'(t)^2
+ (1+\rho(t))^2  \left(  |\theta_1'(t)|^2 \cos^2 \alpha(t) + |\alpha'(t)|^2 + |\zeta'(t)|^2 \sin^2 \alpha(t) \right),
\]
and
\begin{multline*}
\gamma'(t)_{\gamma(t)} = \rho'(t) + (1+\rho(t))
\langle {\gamma^*}'(t),\gamma^*(t)\rangle
\\
= \rho'(t) + i (1+\rho(t))
\left(  \theta_1'(t)\cos^2 \alpha(t)
 -\Im (\langle \zeta(t), \zeta'(t)\rangle) \sin^2 \alpha(t)\right),
\end{multline*}
\[
|\gamma'(t)_{\gamma(t)}|^2 = \rho'(t)^2 +
(1+\rho(t))^2  \left| \theta_1'(t)\cos^2 \alpha(t)
 -\Im (\langle \zeta(t), \zeta'(t)\rangle) \sin^2 \alpha(t) \right|^2.
\]

Because of \eqref{DNTlowest},
\begin{multline}
\label{intest}
k_D(z,w) \gtrsim
\\
\int_0^1
\bigg( \frac1{\sqrt{\rho(t)}} \left| \theta_1'(t)\cos^2 \alpha(t)
 -\Im (\langle \zeta(t), \zeta'(t)\rangle) \sin^2 \alpha(t) \right|
 \\
+ |\theta_1'(t)\cos \alpha(t)| +  |\zeta'(t)\sin \alpha(t)|\bigg) \, dt
+ \int_0^1 \left(|\alpha'(t)|+ \frac{|\rho'(t) |}{\sqrt{\rho(t)}}\right) \, dt
.
\end{multline}
The second integral is comparable to the total variation of $|\alpha|+\sqrt \rho$, so larger than
$\beta+ \sqrt{\eps_1}-\sqrt{\eps_2}$.
If this total variation is larger than $c_1 H_D(z,w)$ ($c_1>0$ to be chosen), we are done,
so assume it is smaller. Since  $\int_0^1 |\alpha'(t)|dt \ge \beta$, then from \eqref{est1hd}
\[
\beta \le c_1  C_2\left(  \beta+ \min \left( \sqrt{|\eta|}, \frac{ |\eta|}{\sqrt{\eps_1}} \right)\right).
\]
Choosing $c_1$ so that $c_1C_2\le \frac12$, we get
$\beta \le 2 c_1  C_2 \min \left( \sqrt{|\eta|}, \frac{ |\eta|}{\sqrt{\eps_1}} \right)$. Therefore
re-applying \eqref{est1hd}, we have
\begin{equation}
\label{est2hd}
H_D(z,w) \le C_2 (1+2c_1) \min \left( \sqrt{|\eta|}, \frac{ |\eta|}{\sqrt{\eps_1}} \right).
\end{equation}
Finally, since the total variation of $|\alpha|+\sqrt \rho$ is bounded by $k_D(z,w)\lesssim H_D(z,w)$,
we can choose $c_1$ so that
\begin{equation}
\label{maxdev}
\max_{t\in [0;1]} (|\alpha(t)|+ \sqrt{\rho(t)})
\le c_1 H_D(z,w) \le \frac12 \min \left( \sqrt{|\eta|}, \frac{ |\eta|}{\sqrt{\eps_1}} \right) \le \frac12 .
\end{equation}

We distinguish  two cases, according to how ``big'',  with respect to an appropriate measure, is the set where the first term
in the integrand is dominant.  More explicitly, let
\begin{multline*}
A:= \bigg\{ t\in [0;1]:
 \left| \theta_1'(t)\cos^2 \alpha(t)
 -\Im (\langle \zeta(t), \zeta'(t)\rangle) \sin^2 \alpha(t) \right|
 \\
\ge  \frac{\sqrt{\rho(t)}}{\max(\sqrt{|\eta|}, \sqrt{\eps_1})} | \theta_1'(t)| \bigg\},
\end{multline*}
then let
\[
\mu:= \frac{\int_A  | \theta_1'(t)| \, dt}{\int_0^1  | \theta_1'(t)| \, dt}.
\]

{\bf Case 1: $\mu\ge \frac12$.}

Then
\begin{multline*}
\int_0^1
\frac1{\sqrt{\rho(t)}} \left| \theta_1'(t)\cos^2 \alpha(t)
 -\Im (\langle \zeta(t), \zeta'(t)\rangle) \sin^2 \alpha(t) \right|
\\
\ge
\frac{1}{\max(\sqrt{|\eta|}, \sqrt{\eps_1})} \int_A  | \theta_1'(t)| \, dt
\ge
\frac{1}{2\max(\sqrt{|\eta|}, \sqrt{\eps_1})} \int_0^1 | \theta_1'(t)| \, dt
\\
\ge
\frac{ \eta}{2\max(\sqrt{|\eta|}, \sqrt{\eps_1})} \gtrsim \min (\sqrt{|\eta|}, \frac{|\eta|}{\sqrt{\eps_1}}).
\end{multline*}

{\bf Case 2: $\mu< \frac12$.}

Then, writing $A^c:= [0;1]\setminus A$, $\int_{A^c}  | \theta_1'(t)| \, dt \ge \frac12 \int_0^1 | \theta_1'(t)| \, dt$.
When $t\in A^c$,
\[
|\zeta'(t)| \sin^2 \alpha(t)  \ge \left( \cos^2 \alpha(t)-  \frac{\sqrt{\rho(t)}}{\max(\sqrt{|\eta|}, \sqrt{\eps_1})} \right)
| \theta_1'(t)| \ge \frac14  | \theta_1'(t)|
\]
because of \eqref{maxdev}.  Thus
\begin{multline*}
\int_0^1 |\zeta'(t)| |\sin \alpha(t)| \, dt
\ge \frac1{\max_{t\in [0;1]} |\sin \alpha(t)|} \frac12  \int_{A^c}  | \theta_1'(t)| \, dt
\\
\gtrsim \frac1{\sqrt \eta} \int_0^1 | \theta_1'(t)| \, dt \gtrsim \sqrt \eta,
\end{multline*}
and we are done.
\end{proof}

\section{Refinements of Theorems \ref{thm} and \ref{smooth}}

Note that if $p$ is a strongly pseudoconvex point of a domain $D$ in $\C^n,$ then
\begin{multline}\label{est}
(1-C\sqrt{\d_D(z)})^2\frac{|X_z|^2}{4\d_D^2(z)}+\frac{C^{-1}|X|^2}{\d_D(z)}\le\kappa^2_D(z;X)
\\
\le
(1+C\sqrt{\d_D(z)})^2\frac{|X_z|^2}{4\d_D(z)}+\frac{C|X|^2}{\d_D(z)},
\end{multline}
where $C>1$ and $z\in D$ near $p$ (see e.g.~\cite[Proposition 1.2]{BB}).

So, similarly to \cite[Theorem 1.1]{BB}, it is natural to extend Theorems \ref{thm} and \ref{smooth}
to large classes of Finsler metrics.

Let $\psi:\R^+\to\R$ be  a function such that $\int_0^1\frac{|\psi(x)|}{x}dx$ exists, and let $C>0.$
For a $\mathcal C^{1,1}$-smooth boundary point $p$ of a domain $D$ in $\C^n,$ consider a
Finsler pseudometric $\mathcal F_D:D\times\C^n\to\R^+_0$ such that for $z\in D$ near $p,$
$$
\mathcal F_D(z;X)=(1+\psi(\d_D(z)))\frac{|X_z|}{2\d_D(z)}+\frac{C|X|}{\sqrt{\d_D(z)}}.
$$
Denote by $f_D$ the integrated form of $\mathcal F_D.$

\begin{prop}\label{fin1} There exists $C'>0$ (depending only on $C$ and $\chi_{D,p}$)
such that
$$
f_D(z,w)\le\log(1+C'A_D(z,w)),\quad z,w\in D\mbox{ near }p.
$$
\end{prop}

The proof follows similar lines to that of Theorem \ref{smooth}. The only point
we have to mention is that the integrability of $\psi(x)/x$ implies that
if $z,z'\in D$ near $p$ with $\pi(z)=\pi(z'),$ then
$$
f_D(z,z')\le\left|\log\frac{\sqrt{\d_D(z')}}{\sqrt{\d_D(z)}}\right|+o(1)
$$
(by integrating $\mathcal F_D$ on the segment $[z,z']$).

Note that Proposition \ref{fin1} has an obvious  analogue for Finsler metrics in a
domain of $\mathbb R^n$ which satisfy the hypothesis of Proposition \ref{fin1} with
the Hermitian product replaced by the ordinary scalar product.

Let now $\mathcal G_D:D\times\C^n\to\R^+_0$ be a Finsler
pseudometric such that for $z\in D$ near $p,$
$$
\mathcal G_D(z;X)=\max\Bigl\{(1+\psi(\d_D(z)))\frac{|X_z|}{2\d_D(z)}, \frac{C|X|}{\sqrt{\d_D(z)}}\Bigr\}.
$$
Denote by $g_D$ the integrated form of $\mathcal G_D.$

\begin{prop}\label{fin2}
If $p$ is strongly pseudoconvex, then there exists $c'>0$ such that
$$g_D(z,w)\ge\log(1+c'A_D(z,w)),\quad z,w\in D\mbox{ near }p.$$
\end{prop}

\begin{proof} Having in mind \eqref{cc},
in the case where $\psi(t)=Ct^s$, $s>0$, \cite[Theorem 1.1]{BB} shows that
$g_D(z,w)\ge g(z,w)-C''$, where $g$ is a quantity which by \eqref{cc} can be seen
to satisfy $g(z,w)= \log(1+A_D(z,w))+O(1)$.
This implies the  bound we want when $A_D(z,w)\gtrsim 1$.
It remains to see that we can extend this to our general case.
The only place where the special form of $\psi$ is used is \cite[Lemma 4.1]{BB},
and retracing the computations in the proof of that Lemma \cite[p. 517]{BB}
we find
$$
g_D(z,w)\ge \left|\log\frac{\sqrt{\d_D(z')}}{\sqrt{\d_D(z)}}\right|
- \int_0 \frac{|\psi(t^2)|}{t} dt,
$$
and this last integral converges by our hypothesis on $\psi$.

Finally, by \eqref{est}
$\mathcal G_D\asymp\kappa_D$. Hence \eqref{2} implies that
$g_D(z,w)\gtrsim  A_D(z,w)$ if $A_D(z,w)\lesssim 1$.
\end{proof}

\noindent{\bf Acknowledgements.} We wish to thank \L ukasz Kosi\'nski for useful
discussions about the topic of this work. We also wish to thank the referee, whose
careful reading allowed us to improve the presentation and correctness of this work.


\begin{thebibliography}{}

\bibitem{BB} Z.M. Balogh, M. Bonk, {\it Gromov hyperbolicity and the Kobayashi
metric on strictly pseudoconvex domains}, Comment. Math. Helv. 75 (2000), 504--533.

\bibitem{BGNT} F. Bracci, H. Gaussier, N. Nikolov, P.J. Thomas, 
{\it Local and global visibility and Gromov hyperbolicity of domains with respect to the Kobayashi distance},
Trans. Amer. Math. Soc. 377 (2024), no. 1, p. 471--493.
	
\bibitem{BNT}  F. Bracci, N. Nikolov, P.J. Thomas, {\it Visibility of Kobayashi geodesics
in convex domains and related properties}, Math. Z. 301 (2022), 2, 2011--2035.

\bibitem{DNT} N.Q. Dieu, N. Nikolov, P.J. Thomas, {\it  Estimates for invariant
metrics near non-semipositive boundary}, J. Geom. Anal. 23 (2013), 598--610.

\bibitem{KNO} \L. Kosi\'nski, N. Nikolov, A.Y. {\"O}kten, {\it Precise estimates
of invariant distances on strongly pseudoconvex domains}, Adv. Math. 478 (2025), 110388.

\bibitem{KNT} \L. Kosi\'nski, N. Nikolov, P.J. Thomas, {\it A Gehring-Hayman inequality for
strongly pseudoconvex domains}, Int. Math. Res. Not. IMRN, no. 11 (2024),  9165--9177.

\bibitem{LPW1} J. Liu, X. Pu,  H. Wang, {\it Bi-H\"older extensions of quasi-isometries on
pseudoconvex domains of finite type in $\C^2$}, J. Geom. Anal., 33:152 (2023).

\bibitem{LPW2} H. Li, X. Pu,  H. Wang, {\it The Gehring-Hayman type theorem on pseudoconvex
domains of finite type in $\C^2$}, Ann. Mat. Pura Appl. 203 (2024), 2785--2799.

\bibitem{NA} N. Nikolov, L. Andreev, {\it Estimates of the Kobayashi and
quasi-hyperbolic distances}, Ann. Mat. Pura Appl. 196 (2017), 43--50.

\bibitem{NOT} N. Nikolov, A.Y. {\"O}kten, P.J. Thomas, {\it Visible $\mathcal C^2$-smooth
domains are pseudoconvex}, Bull. Sci. Math. 197 (2024), 103525.

\bibitem{NP} N. Nikolov, P. Pflug, {\it On the derivatives of the Lempert functions},
Ann. Mat. Pura Appl. 187 (2008), 547--553.

\bibitem{NT1} N. Nikolov, P.J. Thomas, {\it Quantitative localization and comparison of
invariant distances of domains in $\C^n$}, J. Geom. Anal. 33 (2023), article 35.

\bibitem{NT2} N. Nikolov, P.J. Thomas, {\it Boundary regularity for the distance functions,
and the eikonal equation }, J. Geom. Anal. 35 (2025), article 230.

\bibitem{NM} N. Nikolov, M. Trybu\l{}a, {\it The Kobayashi balls of ($\C$-)convex domains},
Monatsh. Math. 177 (2015), 627--635.

\end{thebibliography}
\end{document}